\documentclass[12pt]{amsart}

\usepackage{amsmath} \usepackage{amssymb} \usepackage{amsthm}
\usepackage{amscd}
\usepackage[all]{xy}
\usepackage{xypic}
\usepackage{enumerate}
\usepackage{mathrsfs}
\usepackage{multirow}
\usepackage{bigdelim}
\usepackage[bookmarks=false]{hyperref}
\usepackage{lipsum}

\def\tb{\textbf}

\def\E{{\mathcal{E}}}

\def\G{{\mathcal{G}}}

\def\L{{\mathcal{L}}}

\def\O{{\mathcal{O}}}

\def\Z{{\mathbb{Z}}}
\def\N{{\mathbb{N}}}
\def\Q{{\mathbb{Q}}}
\def\R{{\mathbb{R}}}
\def\C{{\mathbb{C}}}

\def\ol{\overline}

\def\Pic{{\mathrm{Pic}}}

\theoremstyle{plain}
\newtheorem{thm}{Theorem}[section]
\newtheorem{cor}[thm]{Corollary}
\newtheorem{prop}[thm]{Proposition}
\newtheorem{conj}[thm]{Conjecture}

\newtheorem{lem}[thm]{Lemma}
\theoremstyle{definition}
\newtheorem{defn}[thm]{Definition}

\theoremstyle{remark}
\newtheorem{rem}[thm]{Remark}

\newtheorem*{acknowledgement}{Acknowledgments}

\begin{document}

\title{A note on Iitaka's conjecture $C_{3,1}$ in positive characteristics}
\author{Lei Zhang}
\address{Lei Zhang\\School of Mathematics and Information Sciences\\Shaanxi Normal University\\Xi'an 710062\\P.R.China}
\email{lzhpkutju@gmail.com, zhanglei2011@snnu.edu.cn}
\thanks{2010 \emph{Mathematics Subject Classification}: 14E05, 14E30.}
\thanks{\emph{Keywords}: Kodaira dimension; positive characteristics; weak positivity; minimal model.}
\baselineskip = 15pt
\footskip = 32pt

\maketitle
\begin{abstract}
Let $f:X\to Y$ be a fibration from a smooth projective 3-fold to a smooth projective curve, over an algebraically closed field $k$ of characteristic $p >5$. We prove that if the generic fiber $X_{\eta}$ has big canonical divisor $K_{X_{\eta}}$, then
$$\kappa(X)\ge\kappa(Y) + \kappa(X_{\eta}).$$\\
\end{abstract}
\section{Introduction} \label{section:intro}

Throughout this paper, a {\it fibration} means a projective morphism $f:X\to Y$
between varieties such that the natural morphism $\O_Y \to f_*\O_X$ is an isomorphism.

Let $X$ be a projective variety over a field $k$ and $D$ a $\mathbb{Q}$-Cartier divisor on $X$. The $D$-dimension $\kappa(X,D)$ is defined as
\[\kappa(X,D) =\left\{
\begin{array}{llr}
-\infty, \text{ ~~if for every integer } m >0, |mD| = \emptyset;\\
\max \{\dim_k \Phi_{|mD|}(X)| m \in \mathbb{Z}~\text{and}~m>0 \}, \text{ otherwise.}
\end{array}\right.
\]
If $X$ has a regular projective birational model $\tilde{X}$, the Kodaira dimension $\kappa(X)$ of $X$ is defined as $\kappa(\tilde{X}, K_{\tilde{X}})$
where $K_{\tilde{X}}$ denotes the canonical divisor. Kodaira dimension is one of the most important birational invariant in the classification theory.

Iitaka's conjecture states that
\begin{conj}\label{Iit-conj} Let $f:X\rightarrow Y$ be a fibration between smooth projective varieties over an algebraically closed field $k$, with $\dim X = n$ and $\dim Y = m$. Then
  $$C_{n,m}: \kappa(X)\geq \kappa(Y) + \kappa(X_{\eta})$$
where $X_{\eta}$ denotes the generic fiber of $f$.
\end{conj}

When $\mathrm{char} k = 0$, $X_{\eta}$ is a smooth variety over $k(\eta)$, and this conjecture has been widely studied (\cite{Bir09,CH11,Fuj14,Kaw81,Kaw82,Kaw85,Kol87,Vie77,Vie83}).

In positive characteristics this conjecture will be subtler: $X_{\eta}$ is regular but not necessarily smooth over $k(\eta)$, the geometric generic fiber $X_{\bar{\eta}}$ is not even reduced if $f$ is not separable. In \cite{Zha16} a weaker conjecture is proposed: assuming the geometric generic fiber $X_{\bar{\eta}}$ is integral and has a smooth projective birational model $\tilde{X}_{\bar{\eta}}$, then
$$WC_{n,m}: \kappa(X)\geq \kappa(Y) + \kappa(\tilde{X}_{\bar{\eta}}).$$
Note that $\kappa(\tilde{X}_{\bar{\eta}}) \leq \kappa(X_{\eta})$ by \cite[Corollary 2.5]{CZ15}, so $C_{n,m}$ implies $WC_{n,m}$. If granted smooth resolution of singularities in dimension $n$, to prove weak subadditivity $WC_{n,m}$, by Frobenius base changes and resolution of singularities we can reduce to prove $C_{n,m}$ for a fibration with smooth geometric generic fiber (\cite[proof of Corollary 1.3]{BCZ15}).

Up to now, the following have been proved
\begin{itemize}
\item[(i)]{$WC_{n, n-1}$ and $C_{2,1}$ (\cite{CZ15});}
\item[(ii)]{$WC_{3,1}$ (over $\bar{\mathbb{F}}_p, p >5$ by \cite{BCZ15}, over general $k$ with $\mathrm{char}~k >5$ by \cite{Ej15} and \cite{EZ16});}
\item[(iii)]{$C_{3,1}$ under the situation that $K_{X_{\eta}}$ is big, $g(Y) >1$ and $\mathrm{char}~k >5$ (\cite{Zha16}).}
\end{itemize}
Some general results on $C_{n,m}$ are proved under certain technical assumptions in \cite{Pat14, Pat13} and \cite{Zha16}.

This paper aims to prove the following result, which implies $C_{3,1}$ under the situation that $K_{X_{\eta}}$ is big and $\mathrm{char}~k >5$.
\begin{thm}\label{mth}
Let $X$ be a normal $\Q$-factorial projective 3-fold and $f:X\to Y$ a fibration to a smooth projective curve over an algebraically closed field $k$ of characteristic $p>5$. And let $\Delta$ be an effective $\Q$-divisor on $X$ such that $(X, \Delta)$ is klt. Assume that $K_X + \Delta$ is $f$-big. Then
$$\kappa(X, K_X + \Delta) \geq 2 + \kappa(Y).$$
\end{thm}

The proof combines ideas of \cite{EZ16} and \cite{Zha16}.

When the geometric generic fiber $X_{\ol\eta}$ has strongly $F$-regular singularities and $K_X$ is $f$-big, similar result has been proved by Ejiri \cite{Ej15} and
Zhang \cite{Zha16}. Let's review their approaches. In this case, for sufficiently divisible $l$ the sheaf $f_*\omega_{X/Y}^l$ is nef (\cite{Pat14}, \cite{Ej15}). Granted this result, $C_{3,1}$ can be proved easily if either $g(Y) > 1$ or $\deg f_*\omega_{X/Y}^l >0$. To treat the case $\deg K_Y = \deg f_*\omega_{X/Y}^l = 0$, Ejiri \cite{Ej15} uses a deep result of vector bundles on elliptic curves (Proposition \ref{prop:decomp}) and trace maps of relative Frobenius iterations in a very clever way; and Zhang \cite{Zha16} proves that $f$ is birationally isotrivial by use of moduli theory.

In general, though $X_{\ol\eta}$ is integral\footnote{this is because a fibration from a normal variety to a curve is separable by \cite[Lemma 7.2]{Bad01}, so the geometric generic fiber is integral by \cite[Theorem 7.1]{Bad01}}, it may have very bad singularities, so $f_*\omega_{X/Y}^l$ is not necessarily nef. Instead, following observations of  \cite{PSZ13} and \cite{Zha16}, by use of minimal model theory in dimension three,
we can show that, for sufficiently divisible $m$ and $g$, the sheaf $F_Y^{g*}(f_*\O_X((m+1)K_X + m\Delta)\otimes \omega_Y^{-1})$ contains a nef sub-bundle $V$ of rank at least $cm^2$ for some $c>0$ (Corollary \ref{corp}).
By some standard arguments we can show Theorem \ref{mth} under the condition that $\deg K_Y> 0$ or that $\deg V > 0$. To treat the remaining case $\deg K_Y = \deg V = 0$, we will use the approach of \cite[Sec. 4]{EZ16}:  first we construct some effective divisors $B_i \sim m(K_X + \Delta) - L_i$ where $L_i \in \Pic^0(Y)$, then consider the restriction of $B_i$ on the log canonical centers of a carefully chosen dlt pair, finally by use of adjunction formula and abundance of log canonical surfaces due to Tanaka (\cite{Ta15}), we can show that some $L_i$ are torsion, which concludes our theorem easily.


\textbf{Notation and Conventions:} In this paper, we fix an algebraically closed field $k$ of characteristic $p>0$.
A {\it $k$-scheme} is a separated scheme of finite type over $k$. A {\it variety}
means an integral $k$-scheme, and a {\it curve} (resp. {\it surface, $n$-fold}) means
a variety of dimension one (resp. two, $n$).

For the notions of minimal model theory such as lc, klt pairs, flip and divisorial
contractions and so on, please refer to \cite[Sec. 2.5]{Bir16}. By \cite{CP08, CP09, Cut04}, we can always take a log smooth resolution of a pair $(X, \Delta)$ in dimension three.

For a variety $X$, we use $F_X^e: X^{e} \to X$ for the $e$-th absolute Frobenius iteration.

Let $\varphi:X \to T$ be a morphism of schemes and let $T'$ be a $T$-scheme.
Then we denote by $X_{T'}$ the fiber product
$X\times_{T}T'$. For a Cartier
or $\Q$-Cartier divisor $D$ on $X$ (resp. an $\O_X$-module $\G$), the pullback of $D$ (resp. $\G$) to $X_{T'}$
is denoted by $D_{T'}$ or $D|_{X_{T'}}$ (resp. $\G_{T'}$ or $\G|_{X_{T'}}$) if it is well-defined.

\begin{small}
\begin{acknowledgement}
The author thanks BICMR and KIAS, where part of this work was done.
He would like to thank Prof. Chenyang Xu and Dr. Sho Ejiri for many useful discussions. He also thanks the anonymous referee for pointing out some inaccuracies and useful suggestions to improve this note.
The author is supported by grant
NSFC (No. 11401358 and No. 11531009).
\end{acknowledgement}
\end{small}

\section{Preliminaries}\label{section:mmp}
In this section, we recall some technical results which will be used in the sequel.
\subsection{Minimal model of $3$-fold}

\begin{thm}\label{thm:mmp}\samepage
Assume that the base field $k$ has characteristic $p>5$. Let $X$ be a normal $\Q$-factorial quasi-projective 3-fold, $f:X\to Z$ a fibration and $\Delta$ an effective $\Q$-divisor on $X$.

(1) If either $(X,\Delta)$ is klt and $K_X+\Delta$ is pseudo-effective over $Z$, or $(X,\Delta)$ is lc and $K_X+\Delta$ has a weak Zariski decomposition\footnote{i.e., there exists a birational projective morphism $\mu: W \to X$ such that $\nu^*(K_X + \Delta) = P + M$ where $P$ is nef over $Z$ and $M$ is effective},
then $(X,\Delta)$ has a log minimal model over $Z$.

(2) If $(X, \Delta)$ is a dlt pair and $Z$ is a smooth projective curve with $g(Z) \geq 1$, then every step of LMMP \emph{(}\cite[Sec. 3.5-3.7]{Bir16}\emph{)} starting from $(X, \Delta)$ is over $Z$.

(3) If $(X, \Delta)$ is klt and $K_X + \Delta$ is nef and big over $Z$, then $K_X + \Delta$ is semi-ample over $Z$.
\end{thm}
\begin{proof}
For (1) please refer to \cite[Theorem 1.2 and Proposition 8.3]{Bir16}.
(2) follows from cone theorem \cite[Theorem 1.1]{BW14}.
For (3), please refer to \cite{Bir16} or \cite{Xu15}.
\end{proof}

\begin{thm}\label{abd-surf}
Let $X$ be a normal projective surface or a curve over a field $K$ with $\mathrm{char}~K = p >0$, and let $\Delta$ be an effective $\Q$-divisor on $X$. Assume that $(X, \Delta)$ is log canonical and $K_X + \Delta$ is nef. Then $K_X + \Delta$ is semi-ample.
\end{thm}
\begin{proof}
If $X$ is a curve, by considering the pair $(X\times C, \Delta\times C)$ instead, where $C$ is an elliptic curve, we may assume that $X$ is a surface. And for a surface, the assertion is \cite[Theorem 1.1]{Ta15}.
\end{proof}

\subsection{Numerical dimension}
\begin{defn}
Let $D$ be a $\R$-Cartier $\R$-divisor on a smooth projective variety $X$ of dimension $n$.
The numerical dimension $\nu(D)$ is defined as the the biggest natural number $k$ such that
$$\liminf_{m \to \infty}\frac{h^0(\llcorner mD \lrcorner + A)}{m^k} >0 \mathrm{~for~some~ample~divisor}~A~\mathrm{on}~X,$$
if such a $k$ does not exist then define $\nu(D) = -\infty$.
\end{defn}
\begin{rem}
Our notation is slightly different from that of \cite[V, Sec. 2]{Nak04}, where $\kappa_{\sigma}(D)$ is used for the above invariant. Usually, the notation $\nu(D)$ is defined for nef $\R$-Cartier $\R$-divisors as
$$\nu(D) = \max \{k \in \N|D^k \cdot A^{n-k} >0\mathrm{~for~an~ample~divisor}~A~\mathrm{on}~X \}.$$
Over the field of complex numbers $\C$, it is well known that for a nef divisor $D$, $\kappa_{\sigma}(D) = \nu(D)$ (\cite[V, 2.7 (6)]{Nak04}).

In arbitrary characteristics, since smooth alteration exists due to de Jong \cite{J97}, this invariant can be defined for $\R$-Cartier $\R$-divisors on normal varieties by pulling back to a smooth variety, which does not depend on the choices of smooth alterations by \cite[2.5-2.7]{CHMS14}.
\end{rem}

We collect some useful results on numerical dimension in the following proposition.
\begin{prop}\label{num-prop}
Let $X$ be a normal projective variety and $D$ a $\R$-Cartier $\R$-divisor on $X$.

(1) If $D$ is nef, then
$$\nu(D) = \max \{k \in \N|D^k\cdot A^{n-k} >0\mathrm{~for~an~ample~divisor}~A~\mathrm{on}~X \}.$$
If moreover $D$ is effective and $S$ is a normal component of $D$, then
$$\nu(D|_S) \leq \nu(D) -1.$$

(2) For two pseudo-effective $\R$-Cartier $\R$-divisors $D_1, D_2$ on $X$, and two real numbers $a, b >0$, we have
$$\nu(D_1 + D_2) = \nu(aD_1 + bD_2).$$

(3) Let $f: X \to Y$ be a fibration and denote by $F$ the generic fiber of $f$. Then we have easy addition
$$\nu(D) \leq \nu(D|_F) + \dim Y.$$
Moreover if there exists a big divisor $H$ on $Y$ such that $D - f^*H$ is effective and $\nu(D|_F) = \kappa(F, D|_F)$, then
$$\nu(D) =\kappa(X, D) = \nu(D|_F) + \dim Y.$$

(4) Let $\mu: W\to X$ be a generically finite morphism between two normal projective varieties and $E$ an effective $\mu$-exceptional Cartier divisor on $W$. Then
$$\nu(D) = \nu(\mu^*D + E).$$

(5) If $(X, \Delta)$ is a $\Q$-factorial log canonical 3-fold, and $(X', \Delta')$ is a minimal model of $(X, \Delta)$, then
$$\nu(K_X + \Delta) = \nu(K_{X'} + \Delta').$$
\end{prop}
\begin{proof}
Note that in positive characteristics, analogous results to \cite[V, 1.12 and 1.14]{Nak04} have been proved by \cite{CHMS14}. So the first assertion of
(1) follows from the same argument of \cite[V, 2.7 (6)]{Nak04}; and the second assertion follows from the relation \cite[Sec. 1.2]{Deb01}
$$(D|_S)^{k-1}\cdot (A|_S)^{n-k} = D^{k - 1}\cdot A^{n-k} \cdot S \leq D^k\cdot A^{n-k}.$$

(2) is an easy consequence of \cite[V, 2.7 (1)]{Nak04}.

For (3), the inequality is \cite[V, 2.7 (9)]{Nak04}, and the remaining assertion follows from similar arguments of \cite[V, 2.27 (9)]{Nak04} or \cite[Lem. 2.22]{BCZ15}.

(4) is implied by the proof of \cite[2.7]{CHMS14}.

Finally for (5), taking a common smooth log resolution of $(X, \Delta)$ and $(X', \Delta')$, this assertion follows from (4) and standard arguments of minimal model theory.
\end{proof}

\subsection{Covering Theorem}
The result below is [\cite{Iit82}, Theorem 10.5] when $X$ and $Y$ are both smooth, and the proof there also applies when varieties are normal.
\begin{thm}\textup{(\cite[Theorem 10.5]{Iit82})}\label{ct}
Let $f\colon X \rightarrow Y$ be a proper surjective morphism between complete normal varieties.
If $D$ is a Cartier divisor on $Y$ and $E$ an effective $f$-exceptional divisor on $X$, then
$$\kappa(X, f^*D + E) = \kappa(Y, D).$$
\end{thm}

As a corollary we get the following useful result, which also appears in \cite{EZ16}.
\begin{lem}\emph{(}\cite[Lemma 2.3]{EZ16}\emph{)}\label{inj-pic}
Let $g: W \rightarrow Y$ be surjective projective morphism between projective varieties. Assume $Y$ is normal and let $L_1,L_2 \in \Pic^0(Y)$ be two line bundles on $Y$. If $g^*L_1 \sim_{\mathbb{Q}} g^*L_2$ then $L_1 \sim_{\mathbb{Q}} L_2$.
\end{lem}
\begin{proof}
Let $L = L_1 \otimes L_2^{-1}$. Denote by $\sigma: W' \to W$ the normalization and  $g'=g\circ \sigma: W' \to Y$. Then $g'^*L \sim_{\Q} 0$. Applying Theorem \ref{ct} to $g': W' \to Y$ gives that $L \sim_{\Q} 0$, which is equivalent to that $L_1 \sim_{\mathbb{Q}} L_2$.
\end{proof}

\subsection{Vector bundles on elliptic curves}
In this subsection, we recall some facts about vector bundles on elliptic curves, which will be used in the proof of Theorem \ref{mth}.
\begin{thm}\label{thm:facts on vb on ell curve}\samepage
Let $C$ be an elliptic curve, and let $\E_C(r,d)$ be the set of isomorphism classes of indecomposable vector bundles of rank $r$ and of degree $d$.
\begin{itemize}
\item[(1)]\textup{(\cite[Theorem 5]{Ati57})}For each $r>0$, there exists a unique element $\E_{r,0}$ of
$\E_C(r,0)$ with $H^0(C,\E_{r,0})\ne0$. Moreover, for every $\E\in\E_C(r,0)$ there exists an $\L\in\Pic^0(C)$ such that $\E\cong\E_{r,0}\otimes\L$.
\item[(2)]\textup{(\cite[Corollary 2.10]{Oda71})} When the Hasse invariant ${\rm Hasse}(C)$
is nonzero, $F_C^*\E_{r,0}\cong \E_{r,0}$. When ${\rm Hasse}(C)$ is zero,
$F_C^*\E_{r,0}\cong \bigoplus_{1\le i\le\min\{r,p\}}\E_{\lfloor(r-i)/p\rfloor+1,0}$,
where $\lfloor r\rfloor$ denotes the round down of $r$.
\end{itemize}
\end{thm}
\begin{thm}[\textup{\cite[1.4. Satz]{LS77}}]\label{thm:fact on vb on sm curve}\samepage
Let $\E$ be a vector bundle on a smooth projective curve $C$.
If ${F_C^e}^*\E\cong\E$ for some $e>0$, then there exists an \'etale morphism $\pi:C'\to C$ from a smooth
projective curve $C'$ such that $\pi^*\E\cong\bigoplus\O_{C'}$.
\end{thm}

The following result has appeared in \cite{EZ16}, for readers' convenience the proof is included.
\begin{prop}[\textup{\cite[Proposition 2.7]{EZ16}}]\label{prop:decomp}\samepage
Let $\E$ be a vector bundle on an elliptic curve $C$. Then there exists a finite morphism $\pi:C'\to C$ from an elliptic curve $C'$ such that $\pi^*\E$ is a direct sum of line bundles.
\end{prop}
\begin{proof}
We may assume that for every finite morphism $\varphi:B\to C$ from an elliptic curve $B$, $\varphi^*\E$ is indecomposable. Set $d:=\deg\E$ and $r:=\mathrm{rank}\E$. We show that $r=1$. Let $Q\in C$ be a closed point. Replacing $\E$ by $((r_C)^*\E)(-dQ)$, we may assume that $d=0$. Here $r_C:C\to C$ is the morphism given by multiplication by $r$. Hence Theorems \ref{thm:facts on vb on ell curve} and \ref{thm:fact on vb on sm curve} imply that when the Hasse invariant of $C$ is nonzero (resp. zero), there exists an \'etale morphism $\pi:C'\to C$ (resp. an $e>0$) such that $\pi^*\E$ (resp. ${F_C^e}^*\E$) is a direct sum of line bundles. This implies that  $r=1$.
\end{proof}
\subsection{Weak positivity}
\begin{thm}\label{mthp}
Let $f: X \rightarrow Y$ be a fibration from a smooth projective $n$-fold to a smooth projective curve. Let $D$ be a nef, $f$-big and $f$-semi-ample Cartier divisor on $X$. Then there exists a real number $c>0$ such that, for sufficiently divisible $m$ and $g$, the sheaf $F_Y^{g*}f_*\O_X(mD + K_{X/Y})$ contains a
nef sub-bundle of rank at least $cm^{n-1}$.
\end{thm}
\begin{proof}
Let $D_m = mD + K_{X/Y}$. Then $D_m - K_{X/Y}=mD$ is nef, $f$-big and $f$-semi-ample. Applying \cite[Theorem 1.11]{Zha16} (or \cite[Corollary 6.29]{PSZ13} if $D$ is $f$-ample), we conclude that for sufficiently divisible $g$, the sheaf $F_Y^{g*}f_*\O_X(mD + K_{X/Y})$ contains a
nef sub-bundle $S^gf_*\O_X(mD + K_{X/Y})$ (see \cite[2.2.2]{Zha16} for the definition), whose rank equals to $\dim_{k(\ol\eta)} S^0(X_{\ol\eta}, (mD + K_{X/Y})|_{X_{\ol\eta}})$ by  \cite[Proposition 2.5]{Zha16}. Finally since $D|_{X_{\ol\eta}}$ is big, the proof of \cite[Proposition 2.3 (2), (3)]{Zha16} shows $\dim_{k(\ol\eta)} S^0(X_{\ol\eta}, (mD + K_{X/Y})|_{X_{\ol\eta}}) \geq cm^{n-1}$ for some $c>0$, which completes the proof.
\end{proof}
\begin{cor}\label{corp}
Let $X$ be a normal $\Q$-factorial projective 3-fold, $f: X \to Y$ a fibration to a smooth projective curve of genus $\geq 1$, and let $\Delta$ be an effective $\Q$-divisor on $X$. Assume that $(X, \Delta)$ is klt and $K_X + \Delta$ is $f$-big. Then there exists a real number $c>0$ such that, for sufficiently divisible $m$ and $g$, the sheaf $F_Y^{g*}(f_*\O_X((m+1)K_X + m\Delta)\otimes \omega_Y^{-1})$ contains a nef sub-bundle of rank at least $cm^2$.
\end{cor}
\begin{proof}
Running an LMMP for $K_X + \Delta$ over $Y$, since $g(Y) \geq 1$, by Theorem \ref{thm:mmp} we get a minimal model $(X', \Delta')$ and a fibration $f': X'\to Y$ such that, $K_{X'} + \Delta'$ is nef, $f'$-big and $f'$-semi-ample. Take a common log smooth resolution $W$ of $(X', \Delta')$ and $(X, \Delta)$, and fit them into the following commutative diagram
$$\centerline{\xymatrix{& &W\ar[ld]_{\mu}\ar[rd]^{\nu}\ar[dd]^g &\\
&X'\ar[rd]_{f'} & &X\ar[ld]^f\\
& &Y &
}}$$
where $\mu, \nu$ denote the natural birational morphisms and $g = f'\circ \mu$.
Applying Theorem \ref{mthp}, we show that there exists some $c>0$ such that, for sufficiently divisible $m$ and $g$, the sheaf
$$F_Y^{g*}(g_*\O_W(m\mu^*(K_{X'} + \Delta')+ K_{W/Y})) \cong F_Y^{g*}(g_*\O_W(m\mu^*(K_{X'} + \Delta')+ K_{W}) \otimes \omega_Y^{-1})$$
contains a nef sub-bundle of rank at least $cm^2$.

Since $(X', \Delta')$ is a minimal model of $(X,\Delta)$ (\cite[2.7]{Bir16}), there exists an effective $\mu$-exceptional divisor $E$ on $W$ such that
$$E \sim m\nu^*(K_{X} + \Delta) - m\mu^*(K_{X'} + \Delta').$$
Since $X$ and $X'$ have klt singularities, we have $\nu_*\O_W(K_W) = \O_{X}(K_{X})$, $\mu_*\O_W(K_W) = \O_{X'}(K_{X'})$
and thus $\mu_*\O_W(K_W + E) = \O_{X'}(K_{X'})$.
Applying projection formula, we get
\begin{align*}
&g_*\O_W(m\mu^*(K_{X'} + \Delta')+ K_{W}) \\
&\cong f'_*\mu_*\O_W(m\mu^*(K_{X'} + \Delta')+ K_{W} + E)   \\
&\cong f_*\nu_*\O_W(m\nu^*(K_{X} + \Delta) + K_{W})\\
&\cong f_*\O_{X}(m(K_{X} + \Delta)+ K_{X}) \cong f_*\O_X((m+1)K_X + m\Delta).
\end{align*}
Therefore, the sheaf $F_Y^{g*}(f_*\O_X((m+1)K_X + m\Delta)) \otimes \omega_Y^{-1})$
contains a nef sub-bundle of rank at least $cm^2$.
\end{proof}

\section{Proof of Theorem \ref{mth}}\label{section: proof}
Let the notation be as in Theorem \ref{mth}. Assume $g(Y) \geq 1$. By Theorem \ref{thm:mmp}, $(X, \Delta)$ has a log minimal model, so $K_X + \Delta$ is pseudo-effective. Then since $K_X + \Delta$ is $f$-big, we have $\nu(K_X + \Delta) \geq 2$. If $\nu(K_X + \Delta) =3$, then $K_X + \Delta$ is big, so we are done. We assume $\nu(K_X + \Delta) = 2$.

By Corollary \ref{corp}, there exists $c>0$ such that, the sheaf $F_Y^{g*}(f_*\O_X((m+1)K_X + m\Delta)\otimes \omega_Y^{-1})$ contains a nef sub-bundle of rank at least $cm^2$ for sufficiently divisible $m$ and $g$.

Fix sufficiently divisible integers $m$ and $g$. We can assume that $(X, \frac{m}{m+1}\Delta)$ is klt, and that $K_X + \frac{m}{m+1}\Delta$ is $f$-big.  Similarly as above, we have that $K_X + \frac{m}{m+1}\Delta$ is also pseudo-effective and $\nu(K_X + \Delta) = \nu(X, K_X + \frac{m}{m+1}\Delta) = 2$. In turn we conclude that

(i) $\nu(K_X + t\Delta) = 2$ for $t\geq \frac{m}{m+1}$ by Proposition \ref{num-prop} (2).

Note that $\kappa(X, K_X + \frac{m}{m+1}\Delta) \leq \kappa(X, K_X + \Delta)$. Replacing $\Delta$ by $\frac{m}{m+1}\Delta$ and $m+1$ by $m$, we can assume that

(ii) $F_Y^{g*}(f_*\O_X(m(K_X + \Delta))\otimes \omega_Y^{-1})$ contains a nef sub-bundle $V$ of rank $r\geq 3$ on $Y^g$, and the restriction on the generic fiber of $X_{Y^g} \to Y^g$ induces a linear system $|V|_{X_{\eta^{g}}}|$, which defines a generically finite map $X_{\eta^g}  \dashrightarrow \mathbb{P}_{k(\eta^g)}^{r-1}$.

If $g(Y) >1$, then $K_Y$ is big. We will argue as in the proof of \cite[Theorem 4.1]{Zha16}. Consider the flat base change $F_Y^{g+1}: Y'=Y^{g+1} \to Y$. Let $\sigma': X' \to X\times_Y Y' $ be the normalization morphism. We get the following commutative diagram
\[\xymatrix@C=2cm{&X'\ar@/^1.2pc/[rrr]|{\sigma}\ar[r]^{\sigma'}\ar[rd]^{f'} &X\times_Y Y' \ar[rr]^{\pi}\ar[d]^{\bar{f}} &    &X\ar[d]^f\\
&     &Y'=Y^{g+1} \ar[r]^{F_{Y^g}} &Y^{g}\ar[r]^{F_{Y}^{g}}     &Y\\
} \]
where $\bar{f}, \pi$ denote the natural projections, $\sigma = \pi\circ \sigma'$ and $f'= \bar{f} \circ \sigma'$.
Let $D = m(K_X + \Delta) - f^*K_Y$. Then
$$F_{Y^g}^* V \subseteq F_{Y^g}^*F_Y^{g*}f_*\mathcal{O}_X(D) \cong \bar{f}_*\pi^* \mathcal{O}_X(D) \subseteq f'_*\mathcal{O}_{X'}(\sigma'^* \pi^{*}D) = f'_*\mathcal{O}_{X'}(\sigma^*D)$$
where the ``$\cong$'' is from \cite[chap. III Proposition 9.3]{Ha77}.
Since $\deg V \geq 0$, applying Riemann-Roch formula we get
\begin{align*}
h^0(Y', F_{Y^g}^*V\otimes \omega_{Y'}) &= \deg (F_{Y^g}^*V\otimes \omega_{Y'}) + r(1-g(Y')) + h^1(Y', F_{Y^g}^*V\otimes \omega_{Y'})\\
& \geq r(2g(Y') - 2) +  r(1-g(Y')) = r(g(Y') - 1) > 0.
\end{align*}
Then we conclude that $h^0(X', \sigma^*D + f'^*K_{Y'}) > 0$, and thus
\begin{align*}
\kappa(X, K_X + \Delta)  &=  \kappa(X, D + f^*K_Y)  \\
  &= \kappa(X', \sigma^*D + \sigma^*f^*K_{Y}) \cdots \text{by Theorem \ref{ct}}\\
  &= \kappa(X', \sigma^*D + p^{g+1} f'^*K_{Y'}) \\
  &=  \kappa(X', \sigma^*D + f'^*K_{Y'} + (p^{g+1}-1) f'^*K_{Y'}) \\
  &= \nu(\sigma^*D|_{X'_{\eta}}) + \dim Y' = 3  \cdots \text{by Proposition \ref{num-prop} (3)}.\\
\end{align*}

From now on, we assume $Y$ is an elliptic curve. Then $\omega_{X} = \omega_{X/Y}$.
First we prove that
\begin{lem}\label{cl}
If there exists $L \in \Pic^0(Y)$ such that $\kappa(X, k(K_X + \Delta) + f^*L) \geq 1$ for some $k \in \Z^{>0}$, then $\kappa(X, K_X + \Delta) \geq 2$.
\end{lem}
\begin{proof}
Denote by $I: X \dashrightarrow Z$ the Iitaka fibration of $k(K_X + \Delta) + f^*L$, where $Z$ is a normal surface or a curve.
We have a blowing up map $\mu: X'\to X$ such that, there exists a morphism $g':X'\to Z$ factoring through $I$.
We can write that $K_{X'} + \Delta'= \mu^*(K_X + \Delta) + E$ where $\Delta'$ and $E$ are effective divisors having no common components. Then $(X', \Delta')$ is klt. Let $(X'', \Delta'')$ be a minimal model of $(X', \Delta')$ over $Z$ and denote by $g'': X'' \to Z$ the natural morphism. By Theorem \ref{thm:mmp}, there is a morphism $f'': X'' \to Y$ fitting into the following commutative diagram
$$\centerline{\xymatrix{&X\ar[rd]_{f} &X'\ar[l]_{\mu}\ar@{.>}[r]\ar[d]\ar@/^2pc/[rr]^{g'} &X''\ar[dl]^{f''}\ar[r]^{g''} &Z \\
& &Y & &}}$$

Let $G'', G'$ denote the generic fiber of $g'', g'$ respectively. Then

(a) $G''$ is normal and $(G'', \Delta''|_{G''})$ is klt, so $(K_{X''} + \Delta'')|_{G''} \sim_{\Q} K_{G''} + \Delta''|_{G''}$ is semi-ample by Theorem \ref{abd-surf};

(b) since $\kappa(X, k(K_X + \Delta) + f^*L) \geq 1$, applying Theorem \ref{ct} we have
\begin{align*}
\kappa(G'', (k(K_{X''} + \Delta'') + f''^*L)|_{G''}) &= \kappa(G', (k(K_{X'} + \Delta') + \mu^*f^*L)|_{G'}) \\
&= \kappa(G', \mu^*(k(K_X + \Delta) + f^*L)|_{G'}) \geq 0;
\end{align*}

(c) by Proposition \ref{num-prop} (4), we have
$$\nu(k(K_{X''} + \Delta'') + f''^*L) = \nu(K_{X''} + \Delta'') = \nu(K_X + \Delta) = 2.$$

If $\dim Z = 2$, then $(k(K_{X''} + \Delta'') + f''^*L)|_{G''} \sim_{num} 0$, because otherwise it will be big, and then $k(K_{X''} + \Delta'') + f''^*L$ is big by Proposition \ref{num-prop} (3), which contradicts (c) above.  So from (a) and (b), it follows that
$$(k(K_{X''} + \Delta'') + f''^*L)|_{G''} \sim_{\Q} 0 \sim_{\Q} (K_{X''} + \Delta'')|_{G''},$$
thus $f''^*L|_{G''} \sim_{\Q} 0$.
On the other hand, since $K_{X''} + \Delta''$ is big over $Y$,
$G''$ is dominant over $Y$.
Take a general fiber $G_1''$ of $g''$. Then $G_1''$ is dominant over $Y$ and  $f''^*L|_{G_1''} \sim_{\Q} 0$.
Applying Corollary \ref{inj-pic} to the morphism $G_1'' \to Y$ gives that $L \sim_{\Q} 0$. So we are done in this case.

If $\dim Z = 1$, applying Proposition \ref{num-prop} (3), similarly as in the previous case we see that
$$\nu((K_{X''} + \Delta'')|_{G''}) = \nu((k(K_{X''} + \Delta'') + f''^*L)|_{G''}) = 1,$$
so $G''$ is dominant over $Y$.
Denote by $h'': G'' \to C$ the Iitaka fibration of $(K_{X''} + \Delta'')|_{G''}$, and by $H''$ a general fiber of $h''$. Then $(K_{X''} + \Delta'')|_{H''} \sim_{\Q} 0$, and $H''$ is dominant over $Y$ since $K_{X''} + \Delta''$ is big over $Y$.  By (b), we conclude that $(k(K_{X''} + \Delta'') + f''^*L)|_{H''} \sim_{\Q} 0$, so $f''^*L|_{H''} \sim_{\Q} 0$, and in turn $L \sim_{\Q} 0$ by Lemma \ref{inj-pic}.
Applying \ref{num-prop} (3), we conclude that
$$\kappa(X, K_X + \Delta) = \kappa(X'', K_{X''} + \Delta'') = \kappa(G'', (K_{X''} + \Delta'')|_{G''}) + 1 = 2.$$

In conclusion, we complete the proof of this lemma.
\end{proof}

Now let's proceed with the proof.

\underline{Step 1:}
Since $Y$ is a normal curve, we have that $f$ is separable by \cite[Lemma 7.2]{Bad01}, so the geometric generic fiber is integral.

By Proposition \ref{prop:decomp} there exists a flat base change between two elliptic curves $\pi': Y_1 \rightarrow Y^g$ such that $\pi'^*V = \oplus_{i=1}^n \mathcal{L}_i$. Let $\pi = F_Y^g\circ \pi': Y_1 \to Y^g \to Y$.
And let $X_1$ be the normalization of $X\times_Y Y_1$. We get the following commutative diagram
$$\centerline{\xymatrix{ &X_1\ar[d]_{f_1}\ar[rr]^{\pi_1} & &X\ar[d]^f  \\ &Y_1\ar@/_1pc/[rr]|{\pi}\ar[r]^{\pi'} &Y^g\ar[r]^{F_Y^g}  &Y}}$$
where $\pi_1$ and $f_1$ denote the natural projections.
We have that $\pi^*f_*\O_X(m(K_X + \Delta)) \subset f_{1*}\O_{X_1}(\pi_1^*m(K_X + \Delta))$ by \cite[chap. III Proposition 9.3]{Ha77}, thus
$$\pi'^*V = \oplus_{i=1}^r \mathcal{L}_i \subset f_{1*}\O_{X_1}(\pi_1^*m(K_X + \Delta)).$$

If $\deg V > 0$, up to an \'{e}tale base change, we can assume $\deg V > 1$, thus $h^0(X_1, \pi_1^*m(K_X + \Delta)) \geq 2$ by Riemann-Roch formula. Applying Theorem \ref{ct}, we have
$$\kappa(X, K_X + \Delta) =  \kappa(X_1, \pi_1^*(K_X + \Delta)) \geq 1,$$
which implies the proof by Lemma \ref{cl}.

In the following we assume $\deg V = 0$. So by Theorem \ref{thm:facts on vb on ell curve} and \ref{thm:fact on vb on sm curve}, we have that $\mathcal{L}_i \in \Pic^0(Y_1)$.

\underline{Step 2:} By Step 1, we have
$$h^0(X_1, \pi_1^*m(K_X + \Delta) - f_1^*\mathcal{L}_i) \geq 1, ~\mathrm{thus}~\kappa(X_1, \pi_1^*m(K_X + \Delta) - f_1^*\mathcal{L}_i) \geq 0,$$
and if $\mathcal{L}_i = \mathcal{L}_j$ for some $j \neq i$ then the strict inequalities hold.
Since $\pi^*: \mathrm{Pic}^0(Y) \to \mathrm{Pic}^0(Y_1)$ is surjective, there exist $L_i$ such that $\mathcal{L}_i \sim \pi^*L_i$, thus
$$\pi_1^*m(K_X + \Delta) - f_1^*\mathcal{L}_i \sim \pi_1^*(m(K_X + \Delta) - f^*L_i).$$
By Theorem \ref{ct} and Lemma \ref{cl}, we can  assume $\kappa(X, m(K_X + \Delta) - f^*L_i) = 0$.

In the following we only need to show that at least two of $L_i$ are torsion, because then $h^0(X, l(m(K_X + \Delta) - f^*L_i) \geq 2$ for sufficiently divisible $l$,
so we are done by Lemma \ref{cl}.

Replacing $m$ and $L_i$ by their multiplications by a sufficiently divisible integer, we can assume $h^0(X, m(K_X + \Delta) - f^*L_i) = 1$ for every $i= 1,2,\cdots, r$.

\underline{Step 3:}
By Step 2, we get effective divisors
$$B_i \sim m(K_X + \Delta) - f^*L_i.$$
Take a log smooth resolution $\mu: \tilde{X} \rightarrow X$ of the pair $(X, \Delta+\sum_j B_j)$. Denote by $\tilde{f}: \tilde{X} \rightarrow Y$ the natural morphism.
Let $\tilde{B}$ be the reduced divisor with support equal to $\mathrm{supp} \mu_*^{-1}(\Delta+\sum_j B_j) \cup E(\mu)$, where $E(\mu)$ denotes the union of all the exceptional divisors.
Consider the dlt pair  $(\tilde{X}, \tilde{B})$.
There exists an effective $\Q$-divisor $\tilde{B}'$ on $\tilde{X}$ such that
$$K_{\tilde{X}} + \tilde{B} = \mu^*(K_X + \Delta)+ \tilde{B}'$$
Since $(X, \Delta)$ has a minimal model, $K_{\tilde{X}} + \tilde{B}$ has a weak Zariski decomposition.
By Theorem \ref{thm:mmp}, $(\tilde{X}, \tilde{B})$ has a minimal model $(\hat{X}, \hat{B})$ which can be assumed dlt by taking a crepant resolution (\cite[Theorem 1.6]{Bir16}), and there exists a natural morphism $\hat{f}: \hat{X} \to Y$.
By the construction above, the following are true:\\
(1) Since $(X, \Delta)$ is klt, $\mathrm{supp}\tilde{B}' \subset \mathrm{supp} \tilde{B}$.\\
(2) Since $\tilde{B}' \leq \mu^*\sum_j B_j  + t \mu^*\Delta + E \sim_{num} \mu^*k(K_X + s\Delta) + E$ for some $s,t, k>1$ and an effective $\mu$-exceptional divisor $E$, applying Proposition \ref{num-prop} (4) and (5), we conclude
$$2 \leq \nu(K_{\hat{X}} + \hat{B}) = \nu(K_{\tilde{X}} + \tilde{B}) \leq  \nu(\mu^*(k+1)(K_X + s\Delta) + E) = \nu(\mu^*(K_X + s\Delta)) = 2,$$
thus the two ``$\leq$'' above are in fact ``$=$''.\\
(3) For a sufficiently divisible integer $M > 0$ and $0 \leq i \leq r$, we have effective Cartier divisors
\begin{equation*}
\begin{split}
\tilde{B}_i = M\mu^*B_i +  Mm\tilde{B}'   & \sim M(m\mu^*(K_X + \Delta) -\tilde{f}^*L_i) + Mm\tilde{B}' \\
                                          & \sim Mm(K_{\tilde{X}} + \tilde{B}) - M\tilde{f}^*L_i.
\end{split}
\end{equation*}
Pushing forward via $\tilde{X} \dashrightarrow \hat{X}$, we get effective divisors on $\hat{X}$
$$\hat{B}_i \sim Mm(K_{\hat{X}} + \hat{B}) - M\hat{f}^*L_i.$$
\\
(4) Considering the restriction of $\hat{B}_i$ on the generic fiber $\hat{X}_{\eta}$ of $\hat{f}$, we get a linear system $|\hat{V}| \subset |Mm(K_{\hat{X}} + \hat{B})_{\eta}|$, which defines a generically finite map $\hat{X}_{\eta} \dashrightarrow \mathbb{P}_{k(\eta)}^{r-1}$ by assumption (ii) and the construction above.
Let $\hat{C}_{\eta}$ be the fixed part of $|\hat{V}|$, and let $\hat{A}_{i, \eta}  = \hat{B}_{i, \eta} - \hat{C}_{\eta}$. Then $\hat{A}_{i, \eta}$ are nef and big divisors on $\hat{X}_{\eta}$ which are linearly equivalent to each other. Since $(K_{\hat{X}} + \hat{B})_{\eta}$ is nef and big, we have $(K_{\hat{X}} + \hat{B})_{\eta} \cdot \hat{A}_{i, \eta} >0$. Take a reduced and irreducible component $\hat{G}_{1, \eta}$ of $\hat{A}_{1, \eta}$ such that $(K_{\hat{X}} + \hat{B})_{\eta} \cdot \hat{G}_{1, \eta} >0$. Then there exists a divisor among the $\hat{A}_{i, \eta}$, say, $\hat{A}_{2, \eta}$ not containing $\hat{G}_{1, \eta}$. So by $(K_{\hat{X}} + \hat{B})_{\eta} \cdot \hat{A}_{1, \eta} = (K_{\hat{X}} + \hat{B})_{\eta} \cdot \hat{A}_{2, \eta}$, there exists another reduced and irreducible component $\hat{G}_{2, \eta}$ of $\hat{A}_{2, \eta}$ such that $(K_{\hat{X}} + \hat{B})_{\eta} \cdot \hat{G}_{2, \eta} >0$, and that the coefficient of $\hat{G}_{2, \eta}$ in $\hat{A}_{2, \eta}$ is bigger than that in $\hat{A}_{1, \eta}$. 
In turn we get two reduced and irreducible components $\hat{G}_{1}, \hat{G}_{2}$ of $\hat{B}_{1}, \hat{B}_{2}$, whose restrictions on $\hat{X}_{\eta}$ coincide with $\hat{G}_{1, \eta}, \hat{G}_{2, \eta}$ respectively.
If writing that
$$(\clubsuit)~~~~\hat{B}_1 = a_{11}\hat{G}_1 + a_{12}\hat{G}_2 + \hat{G}'_3~\mathrm{and}~ \hat{B}_2 = a_{21}\hat{G}_1 + a_{22}\hat{G}_2 + \hat{G}''_3$$
where  $\hat{G}'_3, \hat{G}''_3 \geq 0$ contain none of $\hat{G}_1, \hat{G}_2$, then the integers $a_{11} >a_{21} \geq 0$ and $a_{22} > a_{12} \geq 0$.

\underline{Step 4:} By the construction of Step 3 (3), we have $\mathrm{supp}\tilde{B}_i \subset \mathrm{supp}\tilde{B}$, thus $\hat{G}_1, \hat{G}_2$ are components of $\hat{B}$. Write that
$$\hat{B} = \hat{G}_1 + \hat{G}'_4.$$
where $\hat{G}'_4$ is an effective divisor.

Since char $k >5$ and $(\hat{X}, \hat{B})$ is dlt, we have that $\hat{G}_1$ is a normal projective surface (\cite[Lemma 5.2]{Bir16}), and the pair $(\hat{G}_1,  \hat{G}'_4|_{\hat{G}_1})$ is lc. Applying Theorem \ref{abd-surf}, we have
$$(K_{\hat{X}} + \hat{B})|_{\hat{G}_1} \sim_{\mathbb{Q}} K_{\hat{G}_1} + \hat{G}'_4|_{\hat{G}_1}$$
is semi-ample. By the construction in Step 3 (4), $(K_{\hat{X}} + \hat{B})|_{\hat{G}_1}$ is relatively big over $Y$.
Applying Proposition \ref{num-prop} (1), we conclude that
$$\nu((K_{\hat{X}} + \hat{B})|_{\hat{G}_1}) = 1.$$
Considering the Iitaka fibration of $(K_{\hat{X}} + \hat{B})|_{\hat{G}_1}$, we get a morphism
$h_1: \hat{G}_1 \to C_1$ to a curve, with every fiber being dominant over $Y$.

Denote by $H_1$ a general fiber of $h_1$. Then $p_a(H_1) \geq 1$. By adjunction
$$0 \sim_{\Q} (K_{\hat{G}_1} +  \hat{G}'_4)|_{H_1} \sim_{\Q} K_{H_1} +  \hat{G}'_4|_{H_1}$$
we conclude that $K_{H_1} \sim_{\Q} \hat{G}'_4|_{H_1} \sim_{\Q} 0$. Then since $\hat{G}'_4 \cap H_1 = \emptyset$, we see that
$$\hat{G}_2 \cap H_1 = \hat{G}'_3 \cap H_1 = \hat{G}''_3 \cap H_1 = \emptyset.$$
Therefore, by $(K_{\hat{X}} + \hat{B})|_{H_1} \sim_{\Q} 0$ and ($\clubsuit$) of Step 3 (4) we have
\begin{equation}
\begin{split}
-a_{21}M\hat{f}^*L_1|_{H_1} &\sim_{\mathbb{Q}} a_{21}(Mm(K_{\hat{X}} + \hat{B}) - M\hat{f}^*L_1)|_{H_1}\\
&\sim_{\mathbb{Q}} a_{21}\hat{B}_1|_{H_1} \sim_{\mathbb{Q}} a_{11}a_{21}\hat{G}_1|_{H_1}\\
&\sim_{\mathbb{Q}} a_{11}\hat{B}_2|_{H_1} \sim_{\mathbb{Q}} -a_{11}M\hat{f}^*L_2|_{H_1}
\end{split}
\end{equation}
which, by Lemma \ref{inj-pic}, implies that
$$a_{21}ML_1 \sim_{\mathbb{Q}} a_{11}ML_2.$$

Applying similar arguments on $\hat{G}_2$ gives
$$a_{22}ML_1 \sim_{\mathbb{Q}} a_{12}ML_2.$$
Finally by the conditions $a_{11} > a_{21}\geq 0$ and $0\leq a_{12}<a_{22}$, we conclude that $L_1 \sim_{\mathbb{Q}} L_2 \sim_{\mathbb{Q}} 0$ and complete the proof.

\end{document}